\documentclass{article}
\usepackage{graphicx} 
\usepackage{mathrsfs}
\usepackage{amsmath}
\usepackage{amsthm}

\usepackage{amsfonts}
\usepackage{amssymb}
\usepackage{latexsym}
\usepackage[all]{xy}

\allowdisplaybreaks[4] \footskip=15pt

\numberwithin{equation}{section} \theoremstyle{plain}
\newtheorem*{thm*}{Main Theorem}
\newtheorem{theorem}{Theorem}[section]
\newtheorem{corollary}[theorem]{Corollary}
\newtheorem*{corollary*}{Corollary}
\newtheorem{lemma}[theorem]{Lemma}
\newtheorem*{lemma*}{Lemma}
\newtheorem{proposition}[theorem]{Proposition}
\newtheorem*{proposition*}{Proposition}

\newtheorem*{remark*}{Remark}
\newtheorem{definition}[theorem]{Definition}
\newtheorem*{definition*}{Definition}

\newtheorem*{example*}{Example}

\newtheorem*{acknowledgements*}{ACKNOWLEDGEMENTS}

\newcommand{\Ext}{\mbox{\rm Ext}}
\newcommand{\Hom}{\mbox{\rm Hom}}
\newcommand{\End}{\mbox{\rm End}}

\newcommand{\Coker}{\mbox{\rm Coker}}

\newcommand{\G}{\mbox{\rm G-dim}}
\newcommand{\pd}{\mbox{\rm pd}}
\newcommand{\gen}{\mbox{\rm gen}}

\newcommand{\add}{\mbox{\rm add}}

\newcommand{\cogen}{\mbox{\rm cogen}}

\newcommand{\f}{\rm fadim}
\newcommand{\m}{\rm Im}
\newcommand{\Tr}{\rm Tr}

\def\mod{\mathop{\rm mod}\nolimits}
\def\mod{\mathop{\rm mod}\nolimits}

\title{WTCSE}

\begin{document}
\begin{center}
{\large  \bf
$\omega$-left approximation dimensions under Stable equivalence}
\footnote {This research was partially supported by NSFC (Grant No. 12061026)
and Foundation for University Key Teacher by Henan Province (2019GGJS204).}
\\ \vspace{0.6cm} {Juxiang Sun$^{a}$, Guoqiang Zhao$^b$, Junling Zheng$^c$}
\footnote {}
\\ \footnotesize {$^{a}$\it School of Mathematics and Statistics, Shangqiu Normal University, Shangqiu 476000, China\\ E-mail: Sunjx8078@163.com}
\\ \footnotesize {$^{b}$\it School of Science, Hangzhou Dianzi University, Hangzhou, 310018, China\\ E-mail: gqzhao@hdu.edu.cn}
\\ \footnotesize {$^{c}$\it  Department of Mathematics, China Jiliang University, Hangzhou, 310018, Zhejiang Province, PR China\\ E-mail: zhengjunling@cjlu.edu.cn}\\
\end{center}

\bigskip
\noindent{ \bf  Abstract.}
In this paper, we investigate some transfer properties of $\omega$-left approximation dimensions of modules 
of stably equivalent Artin algebras having neither nodes nor semisimple direct summands. 
As applications, we give a one-to-one correspondence between basic (Wakamatsu) tilting modules, 
and prove that the Wakamatsu tilting conjecture is preserved under those equivalences.

\vspace{0.2cm}




\section{\bf Introduction}

Let $\Lambda$ be an Artin algebra and let $\omega$ and $X$ be finitely generated left $\Lambda$-modules. Suppose that there is a complex
$$\eta: \quad 0\to X\stackrel{f_1}\to \omega_1\stackrel{f_2}\to \omega_2\stackrel{f_3}\to \cdots\stackrel{f_i}\to \omega_i\stackrel{f_{i+1}}\to \cdots$$
with each $\omega_i\in \add \omega$, where $\add \omega$ is the subclass of $\Lambda$-modules consisting of all modules isomorphic to direct summands of finite copies of $\omega$, such that
$\m f_i\hookrightarrow \omega_i$ is a left $\add \omega$-approximation of $\m f_i$, for all $i$.
Let $\eta_{n}$ denote the truncated complex ending $\omega_n$ obtained from $\eta$.
$X$ is said to {\it have a $\omega$-left approximation dimension $n$}, denoted by l.app$_{\omega}(X)=n$, if $n$ is the largest positive integer such that $\eta_{n}$ is exact. If $\eta$ is exact,
then $X$ is said to {\it have  infinite $\omega$-left approximation dimension}, denoted by l.app $_{\omega}(X)= \infty$.
The $\omega$-left approximation dimension of $_\Lambda\Lambda$ is just the faithful dimension of $\omega$ defined by Buan and Solberg in \cite{BS}, denoted by
$\f_{\Lambda}\omega$, which is used to describe the number of nonisomorphic indecomposable complements of an almost cotilting module.

The notion of $\omega$-left approximation dimensions of modules was introduced by Huang \cite{Hu1}, which plays an important role in homological algebras and relative homological algebras. It is well known that Wakamatsu tilting modules, (relative) torsionfree modules, and modules having generalized Gorenstein dimension zero relative to a Wakamatsu tilting module are characterized in terms of left approximation dimensions (see \cite{BS, Hu1}).

As an essential link between the representation theory of Artin algebras, stable equivalence was introduced by Auslander and Reiten in \cite{AR1}. Two stably equivalent Artin algebras having neither nodes nor semisimple direct summands share many interesting invariants, such as the rigidity dimension, the extension dimension, the global dimension, stable Grothendieck groups, the dominant dimension,  the projective dimension, and so on (see \cite{CG, G, M1} for details). 

The first part of this paper investigates the transfer properties of $\omega$-left approximation dimensions of modules under stable equivalences. One of our main results is the following theorem (see Theorem 3.5). 

\vspace{0.2cm}

{\bf Theorem A}
{\it Let $\Lambda$ and $\Gamma$  be stably equivalent Artin algebras having neither nodes nor semisimple direct summands, and let $F: \underline{\mod \Lambda}\to \underline{\mod \Gamma}$ be the equivalent functor and $F'=\tau_{\Gamma}\circ  F \circ\tau_{\Lambda}^{-1}$.
Let $\omega=X\oplus I\oplus P$ be a $\Lambda$-module satisfying $\Ext_{\Lambda}^{1}(\omega,\omega)=0$,  where $X\in \mod_{\mathcal{I}}\Lambda, I\in \add\mathcal{I}(\Lambda)_{\mathcal{P}}$ and $P$ is a projective-injective $\Lambda$-module, and let $\nu=F'(X)\oplus F(I)\oplus Q$ with $Q$ the direct sum of all nonisomorphic  indecomposable projective-injective $\Gamma$-modules. Suppose that $M=Y\oplus I'\oplus P'$ is a $\Lambda$-module, where $Y\in \mod_{\mathcal{I}}\Lambda, I'\in \add \mathcal{I}(\Lambda)_{\mathcal{P}}$ and $P'$ is a projective-injective $\Lambda$-module. Taking
$N=F'(Y)\oplus F(I)\oplus Q'$ with $Q'$ a projective-injective $\Gamma$-module, then we have
    $$\mathrm{l.app}_{\omega}M=\mathrm{l.app}_{\nu}N.$$}
    
\vspace{0.2cm}

It is well known that tilting modules are central to the tilting theory.
The classical concept of tilting modules was introduced by Brenner and Butler in \cite{BB}, and extended to modules of finite projective dimension by Miyashita \cite{M}. Wakamatsu \cite{W} further generalized this notion to modules of infinite projective dimension, which are commonly called Wakamatsu tilting modules, followed by the terminology in \cite{GRS}. The Wakamatsu tilting conjecture, posed by Beligiannis and Reiten in \cite[Chapter III]{BR}, states that {\it a Wakamatsu tilting module with finite projective dimension is a tilting module.}
This conjecture is significant in the representation theory of Artin algebras, and is closely related to several homological conjectures, such as the finitistic dimension conjecture, the Nakamaya conjecture, the Gorenstein symmetry conjecture, and so on (see \cite{EM, MR, W2, W3} for details).
 Li and Sun \cite{LS} proved that a stronger class of stable equivalences, called {\it stable equivalences of adjoint type}, preserves the partial tilting modules.  In \cite{SZZ}, we improved this result by showing that classical tilting modules and tilted algebras are
preserved under stable equivalence.
Using Theorem A, we give a one-to-one correspondence between basic (Wakamatsu) tilting modules and prove the following results (see Theorem 3.12).

\vspace{0.2cm}

{\bf Theorem B} {\it let $\Lambda$ and $\Gamma$  be stably equivalent Artin algebras having neither nodes nor semisimple direct summands. Then $\Lambda$ satisfies the Wakamatsu tilting conjecture if and only if  $\Gamma$ does so.}

\vspace{0.2cm}

Finally, the invariant properties of relative $n$-torsionfree modules and the generalized G-dimension under stable equivalences are discussed by using Theorem A.

The paper is organized as follows. In Section 2, we provide preliminary definitions and results.
 Section 3 and 4 are devoted to the proofs of Theorems A and B.


\section { \bf Preliminaries}
In this section, we recall some notations and collect some fundamental results.
Throughout this paper, all rings are Artin algebras over a commutative ring $R$, 
and all modules are finitely generated left $\Lambda$-modules.
Let $\Lambda$ be an Artin algebra and $\omega$ a $\Lambda$-module. 
We use $\mod \Lambda$ to denote the category consisting of all finitely generated $\Lambda$-modules, 
and $\add \omega$ to denote the full subcategory of $\mod \Lambda$ consisting of
all modules isomorphic to direct summands of finite copies of $\omega$. 
We denote by $\gen \omega$ (resp. $\cogen \omega$) the full subcategory of $\mod \Lambda$ 
having objects $X$ such that there is an epimorphism $\omega_0\to X$ with $\omega_0\in \add \omega$ 
(resp. such that there is a monomorphism $X\to \omega_0$ with $\omega_0\in \add \omega$).

A homomorphism $f: X\to \omega_0$ with $\omega_0\in \add \omega$ is called {\it a left $\add \omega$-approximation} of $X$,
if $\Hom_{\Lambda}(\omega_0, -)\to \Hom_{\Lambda}(X, -)$ is exact in $\add \omega$. 
And $f$ is called {\it a left minimal $\add \omega$-approximation} of $X$,
if it is also left minimal, that is, $h \in \End X$  is an automorphism whenever $ fh= f$ (see \cite{ARS,ARS1}).

\begin{lemma}
    Let $M$ and  $\omega$ be $\Lambda$-modules and $M=M_1\oplus M_2$. 
    
    $\mathrm{(1)}$(\cite{ARS}) Suppose that $f: M\to \omega_0$ is a left $\add \omega$-approximation of $M$, 
    then there exists  a left minimal   $\add \omega$-approximation $g: M_1\to \omega_1$, such that $\omega_1\in \add \omega_0$.

    $\mathrm{(2)}$ $\mathrm{l.app}_{\omega}M_1\geq \mathrm{l.app}_{\omega}M$.
\end{lemma}

\begin{proof}
(1) See \cite[$\mathrm{P}_7$, Theorem 2.2]{ARS}.

(2) It follows from the definition of $\omega$-left approximation dimension of $M$ and (1) directly.    
\end{proof}

Let $n$ be a positive integer.
Recall from \cite{BB} that a $\Lambda$-module $\omega$ is said to be {\it $n$-tilting}
if the following conditions are satisfied. (1) $\omega$ is self-orthogonal, that is,
$\Ext_{\Lambda}^{\geq 1}(\omega,\omega)=0$; (2) $\pd_{\Lambda}\omega=n< \infty$; (3)  there exists an exact sequence
$0\to _\Lambda \Lambda\to \omega_0\to \omega_1\to \cdots\to \omega_n \to 0$ with $\omega_i\in \add \omega$ for $0\leq i\leq n$. 
$\omega$ is said to be {\it a tilting module}, if it is $n$-tilting for some positive integer $n$. 
Recall from \cite{W} that a $\Lambda$-module $\omega$ is called {\it a Wakamatsu tilting module} 
if it is self-orthogonal and $\f_{\Lambda}\omega=\infty$.

Let $\omega$ be a $\Lambda$-module and $n$ a positive integer.  
The notion of $\omega$-$n$-torsionfree module was introduced by Huang in \cite{Hu1}, 
as a nontrivial generalization of $n$-torsionfree modules defined in \cite{AB}. 
We refer the reader to \cite{Hu1} for the original definition, 
and here we use the following characterization, which is also proved in \cite{Hu1}.

\begin{definition}
    Let $\omega$ be a $\Lambda$-module with $\f_{\Lambda}\omega\geq n+2$. A $\Lambda$-module $X$ is said to be $\omega$-$n$-torsionfree, if $\mathrm{l.app}_{\omega}X=n$.
\end{definition}

In case $_{\Lambda}\omega=_{\Lambda}\Lambda$, the $\omega$-$n$-torsionfree module defined above is just the $n$-torsionfree module defined in \cite{AR2}.

Let $\omega$ be a Wakamatsu tilting $\Lambda$-module with the endomorphism algebra $\Pi=\End_{\Lambda}\omega$. 
Recall from \cite{AR3} that a $\Lambda$-module $X$ is said to {\it have generalized Gorenstein dimension zero relative to $\omega$}, 
denoted by $\G_{\omega}X=0$, if the following data is satisfied. 
 (1) $X$ is $\omega$-reflexive, that is, the evaluation map $\sigma_{X}: X\to \Hom_{\Pi^o}(\Hom_{\Lambda}(X, \omega), \omega)$ 
 via $\sigma_X(x)(f)=f(x)$, for any $f\in \Hom_{\Lambda}(X,\omega)$ and $x\in X$, is an isomorphism; 
 (2) $\Ext_{\Lambda}^i(X,\omega)=0=\Ext_{T^{op}}^i(\Hom_{\Lambda}(X,\omega), \omega)$ for any $i\geq 1$. 
We denote by $\mathcal{G}_{\omega}(\Lambda)$ the subcategory of $\mod \Lambda$ consisting of all modules 
having generalized Gorenstein dimension zero with respect to $\omega$. 
According to \cite[Lemma 5.1]{BLZ}, 
a $\Lambda$-module $X$ has  generalized Gorenstein dimension zero with respect to $\omega$ 
if and only if $X$ is $\omega$-$\infty$-torsionfree and $\Ext_{\Lambda}^{\geq 1}(X,\omega)=0$. 
Clearly, every projective $\Lambda$-module is in $\mathcal{G}_{\omega}(\Lambda)$ for any Wakamatsu tilting $\Lambda$-module $\omega$.

In case $_{\Lambda}\Lambda=_{\Lambda}\omega$, a $\Lambda$-module $G$ having generalized Gorenstein dimension zero 
with respect to $\omega$ is just a $\Lambda$-module having Gorenstein dimension zero defined by Auslander in \cite{AR1}. 
Following the terminology of Enoch and Jenda, 
a module having Gorenstein dimension zero is called Gorenstein projective \cite{EJ}.


\section{ \bf Main Results}

We first recall some notations and results about the stable equivalence of Artin algebras, as detailed in references \cite{AR1, AR2, ARS}.

Let $\Lambda$ be an Artin algebra. By $\mod_{\mathcal{I}}\Lambda$ ({\it resp.} $\mod_{\mathcal{P}}\Lambda$) we denote the
subclass of $\mod \Lambda$ consisting of all $\Lambda$-modules without nonzero injective direct summands
({\it resp.} all $\Lambda$-modules without nonzero projective direct summands). We use $\underline{\mod} \Lambda$ to denote
{\it the stable module category of $\Lambda$ modulo projective module}, which is given by the following data.
The objects are the same as those of $\mod \Lambda$, and for two $\Lambda$-modules $X,Y$ in $\underline{\mod }\Lambda$,
their homomorphism set is given by $\underline{\Hom}_{\Lambda}(X,Y)=\Hom_{\Lambda}(X,Y)/\mathcal{P}(X,Y)$, where $\mathcal{P}(X,Y)$
is the subset of $\Hom_{\Lambda}(X,Y)$ consisting of homomorphisms that factor through a projective $\Lambda$-module.
Dually, we use $\overline{\mod} \Lambda$ to denote the stable module category of $\Lambda$ modulo injective modules.
Let $\tau_{\Lambda}$ be the Auslander-Reiten translation $\mathrm{D}\Tr$. 
Then $ \tau_{\Lambda}: \underline{\mod} \Lambda\to \overline {\mod}\Lambda$ is an equivalence of  additive categories. 
Two Artin algebras $\Lambda$ and $\Gamma$ are said to be {\it stably equivalent}, 
if $\underline{\mod} \Lambda$ and $\underline{\mod} \Gamma$ are equivalent as additive categories.

Let $\Lambda$ and $\Gamma$ be two stably equivalent Artin algebras, 
and let $F: \underline{\mod} \Lambda \to \underline{\mod} \Gamma$ be the stable equivalence functor. 
Then $F'=\tau_{\Gamma}\circ F\circ \tau_{\Lambda}^{-1}: \overline{\mod }\Lambda\to \overline{\mod} \Gamma$
is also an equivalence of additive categories. According to \cite[Section 8]{AR2}, 
it follows that $F$ and $F'$ commute with finite direct sums and induce the following bijections:
$$F: \mod_{\mathcal{P}}\Lambda\to \mod_{\mathcal{P}}\Gamma \qquad \text{and } \qquad F': \mod_{\mathcal{I}}\Lambda\to \mod_{\mathcal{I}}\Gamma$$

Recall that a simple $\Lambda$-module $S$ is said to be {\it a node}, 
if it is neither projective nor injective and the middle term of the almost split
sequence starting at $S$ is projective. 

Let $\Lambda$ be an Artin algebra. An exact sequence $0\to X \stackrel{f}\to Y \stackrel{g}\to Z\to 0$ in $\mod \Lambda$ is said to be {\it minimal}, 
if it does not have a split exact sequence as its direct summand, that is, 
there do not exist isomorphisms $\alpha, \beta, \gamma$ such that the following diagram with exact rows commutes

$$\xymatrix{
  0 \ar[r]& X \ar[d]_{\alpha} \ar[r]^{f} & Y \ar[d]_{\beta} \ar[r]^{g} & Z \ar[d]_{\gamma} \ar[r] & 0  \\
  0 \ar[r] & X_1\oplus X_2 \ar[r]_{\left(\begin{array}{cc}
                f_1 & 0 \\
                 0 & f_2 \\
               \end{array}\right)} & Y_1\oplus Y_2 \ar[r]_{\left(\begin{array}{cc}
                g_1 & 0 \\
                 0 & g_2 \\
               \end{array}\right)} & Z_1\oplus Z_2 \ar[r]& 0   }$$
 where $Y_2\neq 0$ and the sequence $0\to X_2 \stackrel{f_2}\to Y_2 \stackrel{g_2}\to Z_2\to 0$ splits. 
 For a minimal exact sequence $0\to X \stackrel{f}\to Y \stackrel{g}\to Z\to 0$, it is not hard to check
 that $X\in \mod_{\mathcal{I}}\Lambda$ and $Z\in \mod_{\mathcal{P}}\Lambda$.

 In the following,  we always assume that $\Lambda$ and $\Gamma$  are stably equivalent Artin algebras 
 having neither nodes nor semisimple direct summands, and $F: \underline{\mod \Lambda}\to \underline{\mod \Gamma}$ 
 be the equivalent functor and $F'=\tau_{\Gamma}\circ  F \circ\tau_{\Lambda}^{-1}$.
 
\begin{lemma}
Let $F$ and $F'$  be as above.

$\mathrm{(1)}$ There exist one-to-one correspondences
$$F: \add \mathcal{I}(\Lambda)_{\mathcal{P}}\to \add \mathcal{I}(\Gamma)_{\mathcal{P}}, \qquad
F': \add \mathcal{P}(\Lambda)_{\mathcal{I}}\to \add \mathcal{P}(\Gamma)_{\mathcal{I}}.$$

$\mathrm{(2)}$ Suppose that  $X$ is a $\Lambda$-module having neither nonzero projective direct summands nor nonzero injective direct summands. Then we have $F(X)\cong F'(X).$
\end{lemma}

\begin{proof}
    Since $F$ and $F'$ preserve the finite direct sums by \cite[Section 8]{AR2}, the result follows directly from \cite[Lemma 4.10(1)]{CG}.
\end{proof}

Due to \cite[Theorem 1.7]{M1} and Lemma 3.1, the following lemma is direct.

\begin{lemma} Suppose that
 $$ 0\to A \to B\oplus I_1\oplus P\to C\oplus I_2\to 0$$
     is  a minimal exact sequence of $\Lambda$-modules,  where $A,B,C\in \mod_{\mathcal{I}}\Lambda$, $I_1,I_2\in \add\mathcal{I}(\Lambda)_{\mathcal{P}}$
    and $P$ is a projective-injective $\Lambda$-module. Then there exists a minimal exact sequence in $\mod \Gamma$
    $$ 0\to F'(A) \to F'(B)\oplus F(I_1)\oplus Q\to F'(C)\oplus F(I_2)\to 0$$
    where $Q$ is some projective-injective $\Gamma$-module.
\end{lemma}

\subsection{$\omega$-left approximation dimensions of modules under stable equivalences}

In this subsection, we will investigate transfer properties of $\omega$-left approximation dimensions of modules 
under stable equivalences of Artin algebras having neither nodes nor semisimple direct summands, 
and further prove that faithful dimensions of modules also hold under those equivalences. 
 We begin with the following easy lemma, which improves \cite[Proposition 2.2]{M1}.

\begin{lemma}
    Let $A, A'$ be $\Lambda$-modules with $A=A_1\oplus I\oplus P$ and $A'=A'_1\oplus I'\oplus P'$, 
    where $A_1,A'_1\in \mod_{\mathcal{I}}\Lambda$,  $I,I'\in \add\mathcal{I}(\Lambda)_{\mathcal{P}}$ and both $P$ and $P'$ projective-injective $\Lambda$-modules.
    Putting $B=F'(A_1)\oplus F(I)\oplus Q$ and $B'=F'(A_1')\oplus F(I')\oplus Q'$, where both $Q$ and $Q'$  are projective -injective $\Gamma$-modules,
    then,  for a positive integer $n$,  we have $$\Ext_{\Lambda}^n(A,A')\cong \Ext_{\Gamma}^n(B,B').$$
\end{lemma}

\begin{proof}
Note that $A_1\in \mod_{\mathcal{I}}\Lambda$, we can write $A_1=A_2\oplus P_1$,
where $A_2\in \mod_{\mathcal{I}}\Lambda\cap \mod_{\mathcal{P}}\Lambda$ and $P_1\in \add\mathcal{P}(\Lambda)_{\mathcal{I}}$.
Then, by \cite[Section 8]{AR2} and Lemma 3.1(2),  there exist isomorphisms $F'(A_1)\cong F'(A_2)\oplus F'(P_1)\cong F(A_2)\oplus F'(P_1)$, 
where $F(A_2)\in \mod_{\mathcal{P}}\Gamma$
and $F'(P_1)\in \add \mathcal{P}(\Gamma)_{\mathcal{I}}$.
 Thus, by assumption, one gets the following isomorphisms
\begin{align*}
 &\Ext_{\Gamma}^n(B,B')\\
 &= \Ext_{\Gamma}^n(F'(A_1)\oplus F(I)\oplus Q, F'(A_1')\oplus F(I')\oplus Q')\\
& \cong \Ext_{\Gamma}^n (F(A_2)\oplus F'(P_1)\oplus F(I)\oplus Q,  F'(A_1')\oplus F(I')\oplus Q' )\\
&\cong \Ext_{\Gamma}^n (F(A_2)\oplus F(I),  F'(A_1') )\\
&\cong \Ext_{\Gamma}^n (F(A_2\oplus I),  F'(A_1') )  \\
&\cong \Ext_{\Lambda}^n (A_2\oplus I,  A_1')\text{\cite[Proposition 2.2]{M}}\\
&\cong \Ext_{\Lambda}^n (A, A').
\end{align*}
\end{proof}

The following lemma plays an important role in this section.

\begin{lemma}
 Let $\omega=X\oplus I\oplus P$ be  a $\Lambda$-module satisfying $\Ext_{\Lambda}^{1}(\omega, \omega)=0$,
where $X\in \mod_{\mathcal{I}}\Lambda, I\in \add\mathcal{I}(\Lambda)_{\mathcal{P}}$, and $P$  is a projective-injective $\Lambda$-module, and  let $\nu=F'(X)\oplus F(I)\oplus Q$ with $Q$ the direct sum of all nonisomorphic indecomposable projective-injective $\Gamma$-modules. Assume $M\in \mod_{\mathcal{I}}\Lambda$. 

$\mathbf{(1)}$ If there exists an exact sequence in $\mod \Lambda:$
$$0\to M\stackrel{f_1}\to \omega_1 \stackrel{f_2}\to \omega_2 \stackrel{f_3}\to \omega_3\stackrel{f_4}\to \cdots \stackrel{f_n}\to \omega_n\qquad(3.1)$$
with  each $\omega_i\in \add \omega$  and $\m f_i\hookrightarrow \omega_i$  a left $\add \omega$-approximation of $\m f_i$, for any $1\leq i \leq n$.
Then  there is an exact sequence in $\mod \Gamma:$
$$0\to F'(M) \stackrel{g_1}\to \nu_1 \stackrel{g_2}\to \nu_2\stackrel{g_3}\to \nu_3\stackrel{g_4}\to \cdots \stackrel{g_n}\to \nu_n \qquad(3.2)$$
with each $\nu_i\in \add \nu$,  such that  $\m g_i\hookrightarrow \nu_i$  is a left $\add \nu$-approximation of $\m g_i$, for any $1\leq i \leq n$.

$\mathbf{(2)}$ If there exists an exact sequence in $\mod \Lambda$
$$0\to M\stackrel{f_0}\to \omega_0 \stackrel{f_1}\to \omega_1 \stackrel{f_2}\to \omega_2\stackrel{f_3}\to \cdots \stackrel{f_{n}}\to \omega_{n}\to 0$$
with  each $\omega_i\in \add \omega$, then one has an exact sequence in $\mod \Gamma$
$$0\to F'(M)  \to \nu_0 \stackrel{g_1}\to \nu_1\stackrel{g_2}\to \nu_2\stackrel{g_3}\to \cdots \stackrel{g_n}\to \nu_n\to 0$$
with each $\nu_i\in \add \nu$.
\end{lemma}

\begin{proof} We only prove (1), and the proof of (2) is similar.

(1)  For the case of $n=1$, suppose that there is an exact sequence
$$0\to M \stackrel{f_1}\to \omega_1\to T\to 0 \qquad(3.3)$$
with $\omega_1\in \add \omega$, such that $f_1: M\to \omega_1$ is a left $\add \omega$-approximation of $M$.
Since $\Ext_{\Lambda}^1(\omega,\omega)=0$, applying $\Hom_{\Lambda}(-,\omega)$ to (3.3) one has $\Ext_{\Lambda}^1(T, \omega)=0$.
We decompose (3.3) into two exact sequences
$$0\to M_1\to \omega_1^1\to T^1\to 0 \qquad(3.4)$$
and
$$0\to M_2\to \omega_1^2\to T^2\to 0 \qquad(3.5)$$
where (3.4) is minimal and (3.5) is split, such that the following  diagram with exact rows commutes
$$\xymatrix@C=20pt@R=15pt{
    0\ar[r]& M \ar[r]\ar[d]_{\cong }&\omega_1\ar[d]_{\cong }\ar[r]&T\ar[d]_{\cong }\ar[r]&0\\
    0\ar[r]& M_1\oplus M_2 \ar[r]& \omega_1^1\oplus \omega_1^2\ar[r]& T^1\oplus T^2\ar[r]&0}$$
Since $\omega_1^1\in \add \omega$ and $T^1\in \mod_{\mathcal{P}}\Lambda$, we can write  $\omega_1^1=X_1\oplus I_1\oplus P_1$,  where $X_1\in \add X, I_1\in \add I$,
and $P_1$ is a projective-injective $\Lambda$-module, and $T^1=T^1_1\oplus I_2$, 
with $T^1_1\in \mod _{\mathcal{I}}\Lambda$ and $I_2\in \add \mathcal{I}(\Lambda)_{\mathcal{P}}$. By Lemma 3.2, there exists 
a minimal exact sequence in $\mod\Gamma$
$$0\to F'(M_1)\stackrel{g'_1}\to F'(X_1)\oplus F(I_1)\oplus Q_1\to F'(T^1_1)\oplus F(I_2)\to 0$$
where $Q_1$ is a projective-injective $\Gamma$-module.
On the other hand, since  (3.5) is split and $M_2\in \mod_{\mathcal{I}}\Lambda$, one gets $M_2\in \add X$. So, $F'(M_2)\in \add F'(X)\subset\add \nu.$
Thus, one has $ F'(X_1)\oplus F(I_1)\oplus Q_1\oplus F'(M_2)\in \add \nu$ by assumption. 
Combining these results, we obtain an exact sequence
$$0\to F'(M_1)\oplus F'(M_2) \stackrel{\left(\begin{array}{cc}
                g'_1 & 0 \\
                 0 & \mathrm{Id}_{F'(M_2)} \\
               \end{array}\right)}\to (F'(X_1)\oplus F(I_1)\oplus Q_1)\oplus F'(M_2)\to F'(T_1')\oplus F(I_2)\to 0.$$
Because $M\cong M_1\oplus M_2$, one has $F'(M)\cong F'(M_1)\oplus F'(M_2)$ by \cite[Section 8]{AR2}.
Noting that $T^1\in \add T$ and $\Ext_{\Lambda}^1(T,\omega)=0$, we have an isomorphism
               $\Ext_{\Gamma}^1(F'(T_1^1)\oplus F(I_2), \nu)\cong \Ext_{\Lambda}^1(T^1, \omega)=0$ by Lemma 3.3.
Definie  $g_1=\left(\begin{array}{cc}
                g'_1 & 0 \\
                 0 & \mathrm{Id}_{F'(M_2)} \\
               \end{array}\right)$. Then  $g_1$ is a left $\add \nu$-approximation of $F'(M)$.

For the case $n\geq 2$, we proceed by induction on $n$. Set $T_i=\Coker f_i$, for any $1\leq i\leq n$. 
The exact sequence (3.1) induces an exact sequence
$$0\to M \stackrel{f_1}\to \omega_1\to T_1\to 0 \qquad(3.6)$$
where $\Ext_{\Lambda}^1(T_1,  \omega)=0$, because $f_1: M\to \omega_1$ is a left $\add \omega$-approximation of $M$
and $\Ext_{\Lambda}^1(\omega,\omega)=0$.

As above, the exact sequence (3.6) can be decomposed into two exact sequences
$$0\to M_1\to \omega_1^1\to T_1^1\to 0 \qquad(3.7)$$
and
$$0\to M_2\to \omega_1^2\to T_1^2\to 0 \qquad(3.8)$$
where (3.7) is minimal and (3.8) is split, such that the following diagram with exact rows is commutative
$$\xymatrix{
    0\ar[r]& M \ar[r]\ar[d]_{\cong }&\omega_1\ar[d]_{\cong }\ar[r]&T\ar[d]_{\cong }\ar[r]&0\\
    0\ar[r]& M_1\oplus M_2 \ar[r]& \omega_1^1\oplus \omega_1^2\ar[r]& T_1^1\oplus T_1^2\ar[r]&0}$$
Since $\omega_1^1\in \add \omega$ and $T_1^1\in \mod_{\mathcal{P}}\Lambda$, 
we can write $\omega_1^1=X_1\oplus I_1\oplus P_1$, 
where $X_1\in \add X, I_1\in \add I$ and $P_1$ is a projective-injective $\Lambda$-module, 
and $T_1^1=Y_1\oplus J_1$ with $Y_1\in \mod_{\mathcal{I}}\Lambda$ and $J\in \add \mathcal{I}(\Lambda)_{\mathcal{P}}$.
Notice that $T_1\in \cogen \omega$ and $J_1\in \add T_1$, one has $J_1\in \add I$. So, $F(J_1)\in \add F(I)\subset \add \nu.$
Due to Lemma 3.2,  there exists a minimal exact sequence in $\mod \Gamma$
$$0\to F'(M_1)\stackrel{g'_1}\to F'(X_1)\oplus F(I_1)\oplus Q_1\to F'(Y_1)\oplus F(J_1)\to 0 \qquad(3.9)$$
where $Q_1$ is some projective-injective $\Gamma$-module.

On the other hand, since the exact sequence (3.8) is split and $M_2\in \mod_{\mathcal{I}}\Lambda$, 
one gets $M_2\in \add X$. So $F'(M_2)\in \add F'(X)\subset \add \nu$.
Noting that $M\cong M_1\oplus M_2$ with $M\in \mod_{\mathcal{I}}\Lambda$,
one has an isomorphism of $\Gamma$-modules $F'(M)\cong F'(M_1)\oplus F'(M_2)$ from Lemma 3.1(1) and \cite[Section 8]{AR2}.
From the minimal exact sequence (3.9), one obtains an exact sequence
$$0\to F'(M)\stackrel{\left(\begin{array}{cc}
                g'_1 & 0 \\
                 0 & \mathrm{Id}_{F'(M_2)} \\
               \end{array}\right)}\to (F'(X_1)\oplus F(I_1)\oplus Q_1)\oplus F'(M_2)\to F'(Y_1)\oplus F(J_1)\to 0 $$
with $\nu_1=F'(X_1)\oplus F(I_1)\oplus Q_1\oplus F'(M_2) \in \add \nu$ by assumption, 
such that $g_1: F'(M)\to \nu_1$ is a left $\add \nu$-approximation of $F'(M)$, 
because $\Ext_{\Gamma}^1(F'(Y_1)\oplus F(J_1), \nu)\cong \Ext_{\Lambda}^1(T_1^1, \omega)=0$ by Lemma 3.3.

Since  there is a long exact sequence
$$0\to T_1\to \omega_2\stackrel{f_3}\to \omega_3 \stackrel{f_4}\to \cdots \stackrel{f_n}\to \omega_n $$
such that $T_1 \hookrightarrow \omega_2 $ is a left $\add \omega$-approximation of $T_1$ and 
$\m f_i\hookrightarrow \omega_i$ is a left $\add \omega$-approximation of $\m f_i$ for any $3 \leq i \leq n$, 
and $Y_1\in \add T_1$, we obtain a long exact sequence
$$0\to Y_1\stackrel{f'_2}\to \omega'_2\stackrel{f'_3}\to \omega'_3 \stackrel{f'_4}\to \cdots \stackrel{f'_n}\to \omega'_n  $$
with each $\omega'_i\in \add \omega_i$,  such that $\m f'_i\hookrightarrow \omega'_i$ is  a  left $\add \omega$-approximation of $\m f'_i$, for any $2 \leq i \leq n+1$ by Lemma 2.1(1).
By the induction hypothesis, there exists a long exact sequence
$$0\to F'(Y_1) \stackrel{g'_2}\to \nu'_2 \stackrel{g_3}\to \nu_3 \stackrel{g_4}\to \cdots \stackrel{g_n}\to \nu_n$$
such that $ g'_2: F'(Y_1)\hookrightarrow \nu_2$  is a left $\add \nu$-approximation of $ F'(Y_1)$ and  $\m g_i\hookrightarrow \nu_i$  a left $\add \nu$-approximation of $\m g_i$, for any $3\leq i \leq n$, respectively.
Hence, we have an exact sequence
$$0\to F'(Y_1)\oplus F(J_1) \stackrel{\left(\begin{array}{cc}
                g'_2 & 0 \\
                 0 & \mathrm{Id}_{F(J_1)} \\
               \end{array}\right)}\to \nu_2'\oplus F(J_1) \stackrel{g_3}\to \nu_3\to \cdots \stackrel{g_n}\to \nu_n$$
It is not hard to check that $g_2=\left(\begin{array}{cc}
                g'_2 & 0 \\
                 0 & \mathrm{Id}_{F(J_1)} \\
               \end{array}\right): F'(Y_1)\oplus F(J_1) \hookrightarrow \nu'_2 \oplus F(J_1)$ is a left $\add \nu$-approximation of $F'(Y_1)\oplus F(J_1)$.

Taking $g_2=\left(\begin{array}{cc}
                g'_2 & 0 \\
                 0 & \mathrm{Id}_{F(J_1)} \\
               \end{array}\right)$ and $\nu_2=\nu'_2\oplus F(J_1)$, we obtain the desired exact sequence (3.2).
\end{proof}

We now give the main result of this section.

\begin{theorem}
Let $\omega=X\oplus I\oplus P$ be a $\Lambda$-module satisfying $\Ext_{\Lambda}^{1}(\omega.\omega)=0$,  where $X\in \mod_{\mathcal{I}}\Lambda, I\in \add\mathcal{I}(\Lambda)_{\mathcal{P}}$ and $P$ is a projective-injective $\Lambda$-module, and let $\nu=F'(X)\oplus F(I)\oplus Q$ with $Q$ the direct sum of all nonisomorphic  indecomposable projective-injective $\Gamma$-modules. Suppose that $M$ is a $\Lambda$-module with $M=Y\oplus I'\oplus P'$, where $Y\in \mod_{\mathcal{I}}\Lambda, I'\in \add \mathcal{I}(\Lambda)_{\mathcal{P}}$ and $P'$ is a projective-injective $\Lambda$-module. Taking
$N=F'(Y)\oplus F(I')\oplus Q'$ with $Q'$ some projective-injective $\Gamma$-module, then we have
    $$\mathrm{l.app}_{\omega}M=\mathrm{l.app}_{\nu}N.$$
\end{theorem}
\begin{proof}

We only prove that $\mathrm{l.app}_{\omega}M\leq \mathrm{l.app}_{\nu}N,$  the proof of $\mathrm{l.app}_{\omega}M\geq \mathrm{l.app}_{\nu}N$ is similar.

Since $\Ext_{\Lambda}^1(\omega, \omega)=0$ by assumption, we have $\Ext_{\Gamma}^1(\nu, \nu)=0$ by Lemma 3.3.

If $\mathrm{l.app}_{\omega}M=1$, then $M\in \cogen \omega$. So,  $Y\in \cogen \omega$ and $I'\in \add I$. 
Then, by Lemma 2.1(1),  there exists an exact sequence  $0\to Y \stackrel{f}\to \omega_0\to T\to 0$ with $\omega_0\in \add \omega$, such that
$f: Y\to \omega_0$ is a left $\add \omega$-approximation of $Y$.
By Lemma 3.4(1), we have an exact sequence
$$0\to F'(Y)\stackrel{g}\to \nu_0\to L\to 0 \qquad (3.10)$$
with $\nu_0\in \add \nu$, such that $g$ is a left $\add \nu$-approximation of $F'(Y)$. 
Thus, we obtain $\Ext_{\Gamma}^1(L, \nu)=0$ by applying the functor $\Hom_{\Gamma}(-,\nu)$ to the sequence (3.10), 
because of  $\Ext_{\Gamma}^1(\nu, \nu)=0$.
Noting that $I'\in \add I$ and $Q'$ is a projective-injective $\Gamma$-module, we have $F(I')\oplus Q'\in \add \nu$.
Thus, the exact sequence (3.10) yields an exact sequence in $\mod\Gamma$:
$$0\to F'(Y)\oplus (F(I')\oplus Q')\stackrel{\left(\begin{array}{cc}
                g & 0 \\
                 0 & \mathrm{Id}_{(F(I')\oplus Q')} \\
               \end{array}\right)}\to \nu_0\oplus (F(I')\oplus Q')\to L\to 0$$
Define  $g'=\left(\begin{array}{cc}
                g & 0 \\
                 0 & \mathrm{Id}_{F(I')\oplus Q'} \\
               \end{array}\right)$. Then $g': N= F'(Y)\oplus F(I')\oplus Q' \to \nu_0\oplus F(I')\oplus Q'$ is a left $\add \nu$-approximation of $N$,
               because of $\Ext_{\Gamma}^1(L, \nu)=0$.  This means $\mathrm{l. app}_{\nu}N\geq 1.$

Suppose that $\mathrm{l.app}_{\omega}M=n$. By the definition of $\omega$-left approximation dimensions of modules and  Lemma 2.1, 
there exists an exact sequence in $\mod\Lambda$
    $$0\to Y \stackrel{f_1}\to \omega_1\stackrel{f_2}\to \omega_2\stackrel{f_3}\to \cdots \stackrel{f_n}\to \omega_n$$
    with $\omega_i\in \add \omega$,  such that $\m f_i\hookrightarrow \omega_i$ is a left $\add \omega$-approximation of $\m f_i$,
    for any $1\leq i\leq n$.
According  to Lemma 3.4(1),  one has a long exact sequence in $\mod \Gamma$
    $$0\to F'(Y) \stackrel{g'_1}\to \nu'_1 \stackrel{g_2}\to \nu_2 \stackrel{g_3}\to \nu_3\stackrel{g_4}\to \cdots \stackrel{g_n}\to \nu_n$$
with each $\nu_i\in \add \nu$, such that $F'(Y) \hookrightarrow \nu_1$ is a left  $\add V$-approximation of $F'(Y)$ and
$\m g_i\hookrightarrow \nu_i$  a left $\add \nu$-approximation of $\m g_i$, for any $2\leq i \leq n$.

On the other hand, noting that  $I'\in \add \mathcal{I}(\Lambda)_{\mathcal{P}}$, one gets $I'\in \add I.$  So, $F(I')\oplus Q'\in \add \nu $.
Taking $g_1$=$\left(\begin{array}{cc}
                g'_1 & 0 \\
                 0 & \mathrm{Id}_{F(I')\oplus Q'} \\
               \end{array}\right)$ and $\nu_1=\nu_1'\oplus F(I')\oplus Q'$,
then there exists an exact sequence in $\mod\Gamma$
    $$0\to N \stackrel{g_1}\to \nu_1\stackrel{g_2}\to \nu_2\to \stackrel{g_3}\to \nu_3\to \cdots \stackrel{g_n}\to \nu_n$$
with each $\nu_i\in\add \nu$, such that  $\m g_i\hookrightarrow \nu_i$  a left $\add \nu$-approximation of $\m g_i$, for any $1\leq i \leq n$.
This implies $\mathrm{1.app}_{\nu} N\geq \mathrm{l.app}_{\omega} M$.

We can prove $\mathrm{1.app}_{\nu}N\leq \mathrm{l.app}_{\omega}M$ in a similar way.
\end{proof}

The next lemma follows immediately from Lemma 3.1.

\begin{lemma}
    Let $\Lambda=X\oplus P'$, where $X\in \add\mathcal{P}(\Lambda)_{\mathcal{I}}$ and $P'$ is a projective-injective $\Lambda$-module.
    Taking $N=F'(X)\oplus Q$ with $Q$ the direct sum of all nonisomorphic indecomposable projective-injective $\Gamma$-modules.
    Then $N$ is a projective generator for $\Gamma$-modules.
\end{lemma}

Before starting the next proposition, we recall from \cite{BS}: a $\Lambda$-module $\omega$ is said to {\it have faithful dimension $n$},
denoted by $\f_{\Lambda}\omega =n$, if $\mathrm{l.app}_{\omega}\Lambda =n$, where $n$ is a positive integer.
One of the applications of Theorem 3.5 is as follows.

\begin{proposition}
Let $\omega, \nu$ be as in Theorem 3.5. Then we have
$$\f _{\Lambda}\omega= \f _{\Gamma}\nu.$$
\end{proposition}

\begin{proof}
Decompose $\Lambda$ as $\Lambda=X\oplus P'$, where $X\in \add \mathcal{P}(\Lambda)_{\mathcal{I}}$ and $P'$ is a projective-injective $\Lambda$-module. 
Define $N=F'(X)\oplus Q$, where $Q$ is the direct sum of all nonisomorphic indecomposable projective-injective $\Gamma$-modules. By Lemma 3.6, one has $\add N= \add\Gamma$. Hence, it is easy to see that $\mathrm{l.app}_{\nu}\Gamma=\mathrm{l.app}_{\nu}N$.
By Theorem 3.5 and the definition of faithful dimensions of modules, we obtain $\f_{\Lambda}\omega=\mathrm{l.app}_{\omega}\Lambda=\mathrm{l.app}_{\nu}N=\mathrm{l.app}_{\nu}\Gamma=\f_{\Gamma}\nu$ as desired.
\end{proof}

Let $\Lambda$ be an Artin algebra and  $M$  a $\Lambda$-module. If
$$0\to M\to I_1\to I_2\to \cdots \to I_n\to \cdots$$
is a minimal injective resolution of $M$, then we say that $M$ {\em has dominant dimension $n$}, denoted by dom.dim$M=n$, provided that $I_i$ is projective for $1\leq i\leq n$, but $I_{n+1}$ is not projective. An application of Theorem 3.5 is the following corollary, which improves \cite[Theorem 2.2]{M1}.
\begin{corollary}
Let $M$ be a $\Lambda$-module with $M=Y\oplus I'\oplus P'$, where $Y\in \mod_{\mathcal{I}}\Lambda, I'\in \add \mathcal{I}(\Lambda)_{\mathcal{P}}$ and $P'$ is a projective-injective $\Lambda$-module. Taking
$N=F'(X)\oplus F(I')\oplus Q'$ with $Q'$ a projective-injective $\Gamma$-module, we have
$$\mathrm{dom.dim} M=\mathrm{dom.dim} N.$$
Furthermore, we have $$\mathrm{dom.dim} \Lambda=\mathrm{dom.dim} \Gamma.$$
\end{corollary}
\begin{proof}
    Let $P$ be the direct sum of all nonisomorphic indecomposable projective-injective $\Lambda$-modules. It is not difficult to check that $\mathrm{dom.dim}\Lambda= \mathrm{l.app}_{P}\Lambda$ and $\mathrm{dom.dim}M=\mathrm{l.app}_{P}M$, for any $\Lambda$-module $M$.
    This corollary follows directly from Theorem 3.5 and Lemma 3.6.
\end{proof}

\subsection{Wakamatsu tilting conjecture under stable equivalences}
In this subsection, we will give some applications of the results in Subsection 3.1. We establish explicit one-to-one correspondences between basic Wakamatsu tilting modules under stable equivalences of Artin algebras that have neither nodes nor semisimple direct summands, and further demonstrate that the Wakamatsu tilting conjecture maintains its validity under such equivalences.

\begin{proposition}
    Let $\omega=X\oplus I\oplus P$ be a Wakamatsu tilting $\Lambda$-module, where $X\in \mod_{\mathcal{I}}\Lambda$, $I\in\add \mathcal{I}(\Lambda)_{\mathcal{P}}$, and $P$ is a projective-injective $\Lambda$-module. Put $\nu=F'(X)\oplus F(I)\oplus Q$, where $Q$ is the direct sum of all nonisomorphic indecomposable projective-injective $\Gamma$-modules. Then $\nu$  is a Wakamatsu tilting $\Gamma$-module.
\end{proposition}

\begin{proof}
By the definition of a Wakamatsu tilting module, $\omega$ is a Wakamatsu tilting module if and only if $\Ext_{\Lambda}^{\geq 1}(\omega,\omega)=0$ and
$\f_{\Lambda}\omega=\infty$. This proposition follows directly from Lemma 3.3 and Proposition 3.7.
\end{proof}

\begin{proposition}
Let $\omega=X\oplus I\oplus P$ be a $\Lambda$-module  with $X\in \mod_{\mathcal{I}}\Lambda$,
    $I\in \add \mathcal{I}(\Lambda)_{\mathcal{P}}$ and $P$ a projective-injective $\Lambda$-module, and let $\nu=F'(X)\oplus F(I)\oplus Q$, where $Q$ is the direct sum of all nonisomorphic indecomposable projective-injective $\Gamma$-modules. If $\omega$ is a tilting $\Lambda$-module, then $\nu$ is a tilting $\Gamma$-module.
\end{proposition}

\begin{proof}
Since $\omega$ is a tilting $\Lambda$-module by assumption, we have $\Ext_{\Lambda}^{i}(\omega,\omega)=0$ for all $i\geq 1$. According to Lemma 3.3, it follows that $\Ext_{\Gamma}^{\geq 1}(\nu, \nu)=0$.

Since $X\in \mod_{\mathcal{I}}\Lambda$, we can write $X=K\oplus P'$, where $K\in \mod_{\mathcal{P}}\Lambda\cap \mod_{\mathcal{I}}\Lambda$
 and $P'\in \add \mathcal{P}(\Lambda)_{\mathcal{I}}$. Thus, one gets  $\omega=K\oplus P'\oplus I\oplus P$  and $\nu\cong  F'(K)\oplus F'(P')\oplus F(I)\oplus Q$,
  by \cite[Section 8]{AR2}.
We have $F'(K)\cong F(K)$ and both $F'(P')$ and $Q$ are projective $\Gamma$-modules by assumption and Lemma 3.1.
Therefore, $\pd \nu=\pd (F(K)\oplus F(I))=\pd F(K\oplus I)=\pd(K\oplus I)=\pd \omega=n$, for a positive integer $n$, by \cite[Corollary P4150]{M1}.

Putting $\Gamma=Q'\oplus I'$ with $Q'\in \add \mathcal{P}(\Gamma)_{\mathcal{I}}$ and $I'$ a projective-injective $\Gamma$-module, there exists $P'\in \add\mathcal{P}(\Lambda)_{\mathcal{I}}$ such that $Q'=F'(P')$ by Lemma 3.1(1).
Noting that $\omega$ is a tilting $\Lambda$-module by assumption, there exists an exact sequence in $\mod \Lambda:$
$$0\to \Lambda \to \omega_0 \to \omega_1\to \cdots \to \omega_n\to 0$$
with $\omega_i\in \add \omega$.
Due to Lemma 2.1(2), there exists an exact sequence
$$0\to P'\to \omega_0'\to \omega_1'\to \cdots \to \omega_n'\to 0$$
with $\omega_i'\in \add \omega$ for any $0\leq i\leq n$. By Lemma 3.4(2),  one obtains a long exact sequence in $\mod\Gamma$
$$0\to Q'\to \nu_0\to \nu_1\to \cdots \to \nu_n\to 0$$
with $\nu_i\in \add \nu$ for each $0\leq i\leq n.$
Note that $I'\in \add \nu$ by assumption, then we have an  exact sequence
$$0\to \Gamma \to \nu_0\oplus I' \to \nu_1\to \nu_2\to \cdots \to \nu_n \to 0$$

Thus, we obtain that $\nu$ is a tilting $\Gamma$-module by the above discussion.
\end{proof}

 Let $P$ be the direct sum of all nonisomorphic indecomposable projective-injective $\Lambda$-modules, and let $Q$ be the direct
sum of all nonisomorphic indecomposable projective-injective $\Gamma$-modules.
The following correspondence is given as follows:

$\Phi: \mod \Lambda \to \mod \Gamma $

defined by $ \Phi(\omega) = F'(X)\oplus F(I)\oplus Q $,
if $\omega=X\oplus I\oplus P'$ with $X\in \mod_{\mathcal{I}}\Lambda, I\in \add \mathcal{I}(\Lambda)_{\mathcal{P}}$
and $P'$ a projective-injective $\Lambda$-module.

And $\Psi: \mod\Gamma \to \mod \Lambda$ \\
defined by $\Psi(\nu) =F'^{-1}(Y)\oplus F^{-1}(J)\oplus P$,
if $\nu=Y\oplus J\oplus Q'$ with $Y\in \mod_{\mathcal{I}}\Gamma, J\in \add \mathcal{I}(\Gamma)_{\mathcal{P}}$
and $Q'$ a projective-injective $\Gamma$-module.

Recall that a $\Lambda$-module is said to be {\it basic} if it is a direct sum of nonisomorphic indecomposable modules.
We use $\mathcal{WT}(\Lambda)$ to denote the subclass of $\mod\Lambda$ consisting of all nonisomorphic basic Wakamatsu tilting $\Lambda$-modules,
 and use $\mathcal{T}(\Lambda)$ to denote the subclass of $\mod\Lambda$ consisting of all nonisomorphic basic tilting $\Lambda$-modules.
We have the following result.

\begin{theorem}
    The correspondences $\Phi$ and $\Psi$ restrict to one-to-one correspondences between $\mathcal{WT}(\Lambda)$ and $\mathcal{WT}(\Gamma)$,
   $\mathcal{T}(\Lambda)$ and $\mathcal{T}(\Gamma)$, respectively.
\end{theorem}

\begin{proof}
    It follows directly from Proposition 3.9 and Proposition 3.10.
\end{proof}

\begin{theorem}
 $\Gamma$ satisfies the Wakamatsu tilting conjecture if and only if $\Lambda$ does so.
\end{theorem}

\begin{proof} 
Assume that $\Gamma$ satisfies the Wakamatsu tilting conjecture.
Let $\omega$ be a basic Wakamatsu tilting $\Lambda$-module with $\pd_{\Lambda}\omega=n<\infty$.
Decompose $\omega$ as $\omega=X\oplus I\oplus P$, where $X\in \mod_{\mathcal{I}}\Lambda$,
    $I\in \add \mathcal{I}(\Lambda)_{\mathcal{P}}$, and $P$ is the direct sum of all nonisomorphic indecomposable projective-injective $\Lambda$-modules. Define
    $\nu=F'(X)\oplus F(I)\oplus Q$, where $Q$ is the direct sum of all nonisomorphic indecomposable projective-injective $\Gamma$-modules. 
    By Lemma 3.1 and Proposition 3.9, $\nu$ is a basic Wakamatsu tilting $\Gamma$-module.

On the other hand, decompose $X$ as $X=X_1\oplus P_1$ with $X_1\in \mod_{\mathcal{I}}\Lambda\cap \mod_{\mathcal{P}}\Lambda $ and
$P_1\in \add \mathcal{P}(\Lambda)_{\mathcal{I}}$. By Lemma 3.1(2),  one has $F'(X_1)\cong F(X_1)$. Thus, there are isomorphisms
$$\nu \cong F(X_1)\oplus F'(P_1)\oplus F(I_1)\oplus Q\cong F(X_1\oplus I)\oplus F'(P_1)\oplus Q.$$
Noting  that $\pd_{\Gamma}F(X_1\oplus I)=\pd_{\Lambda}(X_1\oplus I)$ by \cite[Corollary]{M1}, then one obtains $\pd_{\Gamma} \nu=\pd_{\Lambda}\omega=n$.
It follows that $\nu$ is a tilting $\Gamma$-module by assumption.
by Theorem 3.11, we conclude that $\omega\cong \Psi(\nu)$ is a basic tilting $\Lambda$-module as regards.

Analagously, we can prove that $\Lambda$ satisfies the Wakamatsu tilting conjecture when $\Gamma$ does.
\end{proof}

\subsection{Relative torsionfree modules}

In this subsection, we will give some applications of subsection 3.1 and 3.2. We will investigate some transfer properties of relative torsionfree modules
and that of modules having generalized Gorenstein dimension zero with respect to a Wakamatsu tilting module.
The following proposition is direct.

\begin{proposition}
Let $n$ be a positive integer, and let  $\omega=X\oplus I\oplus P$ be a $\Lambda$-module satisfying  $\Ext_{\Lambda}^{1}(\omega.\omega)=0$ 
and $\f_{\Lambda}\omega\geq n+2$,  where $X\in \mod_{\mathcal{I}}\Lambda, I\in \add\mathcal{I}(\Lambda)_{\mathcal{P}}$ and $P$ is a projective-injective $\Lambda$-module, and let $\nu=F'(X)\oplus F(I)\oplus Q$ with $Q$ the direct sum of all nonisomorphic  indecomposable  projective-injective $\Gamma$-modules.
 Suppose that  $M$ is  a $\omega$-$n$-torsionfree $\Lambda$-module with $M=Y\oplus I'\oplus P'$, where $Y\in \mod_{\mathcal{I}}\Lambda,
 I'\in \add \mathcal{I}(\Lambda)_{\mathcal{P}}$ and $P'$ is a projective-injective $\Lambda$-module. Taking
$N =F'(Y)\oplus F(I')\oplus Q'$ with $Q'$ a projective-injective $\Gamma$-module,
then  $N$ is a $\nu$-$n$-torsionfree $\Gamma$-module.
\end{proposition}

\begin{proof}
Due to Proposition 3.7,  it follows that $\f_{\Gamma}\nu=\f_{\Lambda}\omega \geq n+2$. And, by Theorem 3.5, one gets $\mathrm{l.app}_{\nu}N=\mathrm{l.app}_{\omega}M=n$.
 This implies that $N$ is a $\nu$-$n$-torsionfree $\Gamma$-module as regards.
\end{proof}

It is well known that a $\Lambda$-$n$-torsionfree module is just the usual $n$-torsionfree module defined in \cite{AB}.
The next result follows immediately from Proposition 3.13.

\begin{corollary}
    Let $M$ be a $n$-torsionfree $\Lambda$-module with $M=Y\oplus I'\oplus P'$, where $Y\in \mod_{\mathcal{I}}\Lambda,
 I'\in \add \mathcal{I}(\Lambda)_{\mathcal{P}}$ and $P'$ is a projective-injective $\Lambda$-module. Taking
$N=F'(Y)\oplus F(I')\oplus Q'$ with $Q'$ some projective-injective $\Gamma$-module,
then  $N$ is a $n$-torsionfree $\Gamma$-module.
\end{corollary}

\begin{proof}
    We take $\Lambda=X\oplus I$, where $X\in \add \mathcal{P}(\Lambda)_{\mathcal{I}}$ and $I$ is a projective-injective $\Lambda$-module. Putting $\nu=F'(X)\oplus Q$
    where $Q$ is the direct sum of all nonisomorphic indecomposable projective-injective $\Gamma$-modules. By Proposition 3.13 and Lemma 3.6, one has $N$ is a $\nu$-$n$-torsionfree
    $\Gamma$-module and $\add \nu =\add \Gamma$. This implies that $N$ is an $n$-torsionfree $\Gamma$-module as desired.
\end{proof}


\begin{corollary}
Let $\omega =X\oplus I\oplus P$ be a Wakamatsu tilting $\Lambda$-module with $X\in \mod_{\mathcal{I}}\Lambda$,
$I\in \add \mathcal{I}(\Lambda)_{\mathcal{P}}$ and $P$ a projective-injective $\Lambda$-module, and let $\nu=F'(X)\oplus F(I)\oplus Q$,
where $Q$ is the direct sum of all nonisomorphic indecomposable projective-injective $\Gamma$-modules.
Suppose that $M=Y\oplus I'\oplus P'$ is  a $\omega$-$\infty$-torsionfree $\Lambda$-module with $Y\in \mod_{\mathcal{I}}\Lambda$, $I'\in \add\mathcal{I}(\Lambda)_{\mathcal{P}}$ and $P'$ a projecitve-injective $\Lambda$-module. Taking $\nu=F'(Y)\oplus F(I')\oplus Q'$,
where $Q'$ is a projective-injective $\Gamma$-module, then $N$ is a $\nu$-$\infty$-torsionfree as a $\Gamma$-module.
\end{corollary}

\begin{proof}
    According to Proposition 3.9, one has that $\nu$ is a Wakamatsu tilting $\Gamma$-module. Thus, it follows directly from Proposition 3.13.
\end{proof}

\begin{theorem}\label{main-theorem}
Let $\omega,\nu, M$ and $N$ be as in Corollary 3.15.  If $\G_{\omega}M=0$, then $\G_{\nu}N=0$.
\end{theorem}

\begin{proof}
    Due to \cite[Lemma 5.1]{BLZ}, one has that $\G_{\omega}M=0$ if and only if
   $\Ext_{\Lambda}^{\geq 1}(M,\omega)=0$ and $M$ is $\omega$-$\infty$-torsionfree.  
The theorem follows from Lemma 3.3 and Corollary 3.15.
\end{proof}

In case $_{\Lambda}\omega={_\Lambda}\Lambda$,  a $\Lambda$-module $M$ with $\G_{\omega}M=0$ is Gorenstein projective defined in \cite{EJ}. The next corollary is immediately from Theorem 3.16.

\begin{corollary}  Let $M$ and $N$ be as in Theorem \ref{main-theorem}.
    If $M$ is Gorenstein projective, then $N$ is so.
\end{corollary}



\end{document}